\documentclass[11pt]{article}
\usepackage{floatrow}
\newfloatcommand{capbtabbox}{table}[][\FBwidth]
\usepackage{enumerate}
\usepackage{amsmath,cite,colordvi,xcolor,color}
\usepackage{amsthm}
\usepackage{amssymb}
\usepackage{amsfonts}
\usepackage{latexsym}
\usepackage{tikz}
\usepackage{float}
\usepackage{mathrsfs}[12pt]
\usepackage{multirow,multicol,subfigure}
\usepackage[colorlinks,urlcolor=purple,linkcolor=red,anchorcolor=green,citecolor=blue]{hyperref}
\usepackage{makecell}

\newtheorem{thm}{Theorem}[section]

\newtheorem{cor}[thm]{Corollary}
\newtheorem{lem}[thm]{Lemma}

\theoremstyle{definition}

\newtheorem{rem}[thm]{Remark}

\numberwithin{equation}{section}

\makeatletter \@addtoreset{equation}{section} \makeatother

\tikzstyle{every node}=[circle,inner sep=1pt,fill=white!60]
\tikzstyle{tn}=[shape=circle, draw, color=black!70]

\begin{document}

\begin{center}
{\bf \large
Congruences for the Coefficients of the Powers of the Euler Product
}

 \vspace{15pt}
{\small  Julia Q.D. Du$^{1}$, Edward Y.S. Liu$^{2}$ and Jack C.D. Zhao$^{3}$

$^{1}$Center for Applied Mathematics \\
Tianjin University, Tianjin 300072, P.R. China\\[6pt]

$^{2,3}$Center for Combinatorics, LPMC, \\
Nankai University, Tianjin 300071, P.R. China\\[6pt]

$^{1}$qddu@tju.edu.cn, $^{2}$liu@mail.nankai.edu.cn, $^{3}$cdzhao@mail.nankai.edu.cn
}

\end{center}

\noindent{\bf Abstract.} Let $p_k(n)$ be given by the $k$-th power of the Euler Product $\prod _{n=1}^{\infty}(1-q^n)^k=\sum_{n=0}^{\infty}p_k(n)q^{n}$.
By
investigating the properties of the modular equations of the second and the third order under
the Atkin $U$-operator,
we determine the generating functions of $p_{8k}(2^{2\alpha} n +\frac{k(2^{2\alpha}-1)}{3})$ $(1\leq k\leq 3)$
and $p_{3k}(3^{2\beta}n+\frac{k(3^{2\beta}-1)}{8})$ $(1\leq k\leq 8)$
in terms of some linear recurring sequences.
Combining with a result of Engstrom about the periodicity of linear recurring sequences modulo $m$,
we obtain infinite families of congruences for $p_k(n)$  modulo any $m\geq2$,
where $1\leq k\leq 24$ and $3|k$ or $8|k$.
Based on these congruences for $p_k(n)$,
infinite families of congruences for many partition functions such as
the overpartition function,
$t$-core partition functions and $\ell$-regular partition functions
are easily obtained.



\section{Introduction}
 This paper is devoted to the congruence properties of $p_k(n)$ modulo arbitrary integer $m\geq 2$,
where $p_k(n)$ is defined by the $k$-th power of the Euler Product
\begin{align}\label{Def_Euler_Product}
  f_1^k:=\prod_{n=1}^\infty(1-q^n)^k=\sum_{n=0}^\infty p_k(n)q^n.
\end{align}
When $k=1$, $f_1$ is called the Euler Product,
and by Euler's pentagonal number theorem \cite{Euler}, \cite[pp.~11]{Andrews-book},
we see that
\[p_1(n)=	
  \begin{cases}
    (-1)^m, & \hbox{if $n=m(3m\pm1)/2$;} \\
    0, & \hbox{otherwise.}
  \end{cases}
\]
When $k=-1$, $p_{-1}(n)$ is the ordinary partition function $p(n)$,
which counts the number of ways of writing $n$ as a sum of a non-increasing sequence of  positive integers.

Recently, the congruences for $p_k(n)$ have drawn much more attention
since many partition functions possess the same congruence properties as  $p_k(n)$,
such as $\ell$-regular partition functions,
the overpartition function, $t$-core partition functions.
Recall that for any $\ell\ge 2$,
a partition is called an
\emph{$\ell$-regular partition }if none of its parts is divisible by $\ell$.
Let $b_\ell(n)$ denote the number of $\ell$-regular partitions of $n$,
and the generating function of $b_\ell(n)$ is given by
\[\sum_{n=0}^{\infty}b_\ell(n)q^n=\prod_{n=1}^{\infty}
\frac{(1-q^{\ell n})}{(1-q^n)}\]
with the convention that $b_\ell(0)=1$.
For any prime $\ell$ and $j\geq 0$,
it is easy to see that
\[\sum_{n=0}^{\infty}b_{\ell^j}(n)q^n\equiv f_1^{\ell^j-1}\pmod{\ell}.\]
Combining with \eqref{Def_Euler_Product},
we have for any $n\geq0$,
\[b_{\ell^j}(n)\equiv p_{\ell^j-1}(n)\pmod{\ell}.\]
In \cite{Cui-Gu-Huang-2016},
Cui, Gu and Huang obtained some infinite families of congruences for $p_k(n)$ ($1\leq k\leq 24$) modulo 2 or 3 by using the modular equations of the fifth and the seventh order.
To be specific,
they established the following congruences
\[p_k\left(\ell^{t\alpha}n+\frac{(24i+\ell k)\cdot \ell^{t\alpha-1}-k}{24}\right)\equiv 0\mod {2 ~\text{or}~3},\]
where $1\leq k\leq 24$,
$\alpha\geq 1$ and $\ell = 5$ or $7$,
the positive integer $t$ is  determined by $k$, $\ell$ and the modulus 2 or 3.
Moreover, they also obtained many infinite
families of congruences for $\ell$-regular partition functions
and generalized Frobenius partition functions.
For example,
they proved that
for any $\alpha\geq1$, $n\geq0$ and $1\leq i\leq4$,
\[b_{16}\left(5^{4\alpha}n+i\cdot5^{4\alpha-1}+\frac{3\cdot 5^{4\alpha+1}-15}{24}\right)\equiv 0\pmod{2}.\]
Later,
Xia \cite{Xia2017}  used the modular equations of the  fifth and the seventh,
as well as the thirteenth order to
establish the generating functions of
\[p_k\left(\ell^{2\alpha+1}n+\frac{k(\ell^{2\alpha+2}-1)}{24}
-\ell^{2\alpha+1}t(\ell,k)\right),\]
where $1\leq k\leq 24$, $\alpha\geq 0$ and $\ell=5$, $7$ or $13$ and $t(5,k)=\lfloor\frac{k}{5}\rfloor$, $t(7,k)=\lfloor\frac{2k}{7}\rfloor$ and $t(13,k)=\lfloor\frac{7k}{13}\rfloor$, respectively.
Based on the above generating functions,
he derived that
for a fixed prime $\ell\in\{5,7,13\}$, $1\leq k \leq 24$ and any $m\geq2$,
if there exists $r(k,m)$ satisfying some restricted conditions,
then for any $\alpha\geq1$, $n\geq0$ and $1\leq i\leq \ell-1$,
\begin{align*}
  p_k\left(\ell^{2r(k,m)\alpha} n +
  \frac{(24i+\ell k)\cdot\ell^{2r(k,m)\alpha-1}-k}{24}\right)
  \equiv0
  \pmod{m}.
\end{align*}
Moreover,
he obtained many infinite families of congruences for
$\ell$-regular partition functions,
partition functions related to mock theta functions
and generalized Frobenius partition functions.
For example,
he proved that for any $\alpha\geq1$, $n\geq0$ and $1\leq i\leq12$,
\begin{align*}
  b_{23}\left(13^{22\alpha} n +\frac{(12i+143)\cdot13^{22\alpha-1}-11}{12}\right)\equiv 0\pmod{23}.
\end{align*}

In this paper,
we determine the generating functions of $p_{8k}(2^{2\alpha} n +\frac{k(2^{2\alpha}-1)}{3})$ $(1\leq k\leq 3)$
and $p_{3k}(3^{2\beta}n+\frac{k(3^{2\beta}-1)}{8})$ $(1\leq k \leq 8)$  based on the modular equations of the second and the third order.
Furthermore, combining with a result of Engstrom \cite{Engstrom1931}
about the periodicity of linear recurring sequences modulo $m$,
we obtain infinite families congruences for $p_{8k}(n)$ and $p_{3k}(n)$ with respect to arbitrary modulus $m\geq2$.

To state our results more precisely,
we adopt the notion
\[f_t:=\prod_{n=1}^{\infty}(1-q^{nt})\]
for any integer $t\geq1$.
Note that the case of $t=1$ is coincident with the Euler Product.
By means of the modular equations of the second and the third order,
we derive the following generating functions.

\begin{thm}\label{a8k_GF}
For any $\alpha\geq0$ and $1\leq k \leq 3$,
we have
\begin{align}\label{a8k_GF_1-eq}
 \sum\limits_{n=0}^\infty p_{8k}\Big(2^{2\alpha} n +\frac{k(2^{2\alpha}-1)}{3}\Big)q^n
 = A_k(2\alpha)f_1^{8k} + B_{k}(2\alpha)q^{\lfloor \frac{k}{2}\rfloor}{f_1^{8k-24}}{f_2^{24}}
\end{align}
and
\begin{equation}
\begin{split}
\sum\limits_{n=-\left\lfloor k/2\right\rfloor}^\infty p_{8k}\left(2^{2\alpha+1} n +\frac{k(2^{2\alpha+2}-1)}{3}\right)q^n
=A_k(2\alpha+1)f_2^{8k} + B_{k} (2\alpha+1)q^{-\lfloor\frac{k}{2}\rfloor}f_1^{24}f_2^{8k-24},\label{a8k_GF_2}
\end{split}
\end{equation}
 where the coefficients
$A_k(\alpha)$ and $B_k(\alpha)$
satisfy the same recurrence relation
\begin{equation}\label{2-recurring-relations}
\begin{split}
A_k(\alpha+4) &= f(k)A_k(\alpha+2) + g(k)A_k(\alpha).
\end{split}
\end{equation}
The values of $f(k)$ and $g(k)$  $(1\leq k\leq 3)$
and the initial values of $A_k$ and $B_k$ $(1\leq k\leq 3)$ are listed in Table \ref{coeff_initial}.
\end{thm}

\begin{thm}\label{a3k_GF}
 For any $\beta\geq 0$ and $1\leq k\leq 8$, we have
\begin{align}\label{a3k_gf}
    \sum_{n=0}^\infty p_{3k}\left(
    3^{2\beta}n+\frac{k(3^{2\beta}-1)}{8}\right)q^n
    =C_{k}(2\beta) f_1^{3k}+\sum_{i=1}^{\left\lfloor k/3 \right\rfloor}
    D_{k,i}(2\beta)q^{i}f_1^{3k-12i}f_3^{12i}
\end{align}
and
\begin{equation}\label{a3k_gf1}
\begin{split}
\sum_{n=-\lfloor k/3 \rfloor}^\infty p_{3k}\left(
    3^{2\beta+1}n+\frac{k(3^{2\beta+2}-1)}{8}\right)q^n
    = C_{k}(2\beta+1) f_3^{3k}+
  \sum_{i=1}^{\left\lfloor k/{3}\right\rfloor}D_{k,i}(2\beta+1)
  q^{-i} f_1^{12i}f_3^{3k-12i},
\end{split}
\end{equation}
where the coefficients  $C_k(\beta)$ and $D_{k,i}(\beta)$
satisfy the same recurrence relation
\begin{equation}
\begin{split}
\label{3-recurring-relations}
C_k(\beta+4) &= h(k)C_k(\beta+2) + r(k)C_k(\beta).
\end{split}
\end{equation}
The values of $h(k)$ and $r(k)$ $(1\leq k\leq 8)$
and the initial values of $C_k$ and $D_{k,i}$ $(1\leq k\leq 8)$
are listed in Table \ref{coeff_initial-3}.
\end{thm}

Utilizing a result of Engstrom \cite{Engstrom1931},
we prove that for any $m\geq2$,
there always exist  positive integers $\alpha$ and $\beta$
satisfying
$B_k(2\alpha-1)\equiv 0\pmod m$ and
$D_{k,i}(2\beta-1)\equiv 0\pmod m$ for $1 \leq i \leq \left\lfloor\frac{k}{3}\right\rfloor$
in Theorem \ref{existence_of_mu_nu}.
Therefore we may define $\mu_m(k)$ and $\nu_m(k)$ by
\begin{align}\label{def_mu}
  \mu_m(k)=\min\{\alpha\geq1\mid B_k(2\alpha-1)\equiv 0\pmod m \},
\end{align}
and
\begin{align}\label{def_nu}
   \nu_m(k)=\min\left\{\beta\geq1\mid D_{k,i}(2\beta-1)\equiv 0\pmod m\ \text{ for }\ 1 \leq i \leq \left\lfloor\frac{k}{3}\right\rfloor\right\}.
\end{align}
In addition,
let $0\leq c_1\leq m-1$ be the integer such that $c_1\equiv A_k(2\mu_m(k)-1)\pmod m$,
and $0\leq c_2\leq m-1$ be the integer such that $c_2\equiv C_k(2\nu_m(k)-1)\pmod m$.

We can establish the following infinite families of congruences for $p_{8k}(n)$ and $p_{3k}(n)$ modulo $m$ with the aid of generating functions \eqref{a8k_GF_1-eq} and \eqref{a3k_gf}.
\begin{thm}\label{f1_8k_congruence}
 For any  $m\geq 2$ and $1\leq k \leq 3$,
 the following statements hold:
\begin{description}
\item{$(\mathrm{i})$}
If $c_1^\alpha \not\equiv 0 \pmod m$ for any $\alpha\ge 1$,
then for any $n\ge 0$,
\begin{align*}
p_{8k}\left(2^{2\mu_m(k)\alpha} n + \frac{(2k+3)\cdot 2^{2\mu_m(k)\alpha-1}-k}{3}\right) \equiv 0 \pmod m.
\end{align*}
\item{$(\mathrm{ii})$}
If there exists a positive integer $\alpha$ such that $c_1^\alpha \equiv 0 \pmod m$,
then for any $n\ge 0$,
\[p_{8k}\left(2^{2\mu_m(k)\alpha-1}n+\frac{k(2^{2\mu_m(k)\alpha}-1)}{3}\right) \equiv 0 \pmod m.\]
\end{description}
\end{thm}

\begin{thm}\label{3-congruence}
For any $m\geq 2$ and $1\leq k\leq 8$,
the following statements hold:
\begin{description}
\item{$(\mathrm{i})$}
If $c_2^\beta \not\equiv 0 \pmod m$ for any $\beta\ge 1$,
then for any $n\ge 0$ and $i=1$ or $2$,
\begin{align*}
p_{3k}\left(3^{2\nu_m(k)\beta} n +   \frac{(3k+8i)\cdot3^{2\nu_m(k)\beta-1}-k}{8}\right) \equiv 0 \pmod m.
\end{align*}

\item{$(\mathrm{ii})$}
If there exists a positive integer $\beta$
such that $c_2^\beta \equiv 0 \pmod m$,
then for any $n\ge 0$,
\[p_{3k}\left(3^{2\nu_m(k)\beta-1}n+\frac{k(3^{2\nu_m(k)\beta}-1)}{8}\right) \equiv 0 \pmod m.\]
\end{description}
\end{thm}

It is worth mentioning that there are another two approaches
to obtain some cases of
Theorem \ref{f1_8k_congruence} and Theorem \ref{3-congruence}.
For the first one,
the case of $k=1$ in Theorem \ref{f1_8k_congruence}
and the case of $k=2$ in Theorem \ref{3-congruence}
can also be obtained from
Newman's Theorem \cite[Theorem 1]{Newman}.
For the second one,
the cases of $(k,m)=(1,\geq 2)$ in Theorem \ref{f1_8k_congruence} and the cases of $(k,m)=(2,\geq 2)$,
$(k,m)=(3,12)$, $(k,m)=(4,12)$, $(k,m)=(5,1836)$, $(k,m)=(7,53028)$  in Theorem \ref{3-congruence}
can be deduced
with the aid of eigenforms.
For more details,
please refer to \cite{Cui-Gu-Huang-2016} and \cite{Ono-2004}.

As applications of
Theorem \ref{f1_8k_congruence} and Theorem \ref{3-congruence},
we obtain many infinite families of congruences for different kinds of partition functions.
For example,
we derive some congruences for $\ell$-regular partition functions modulo
$2$, $3$, $5$, $17$ and $19$,
some of which are stated as follows:
for any $\alpha\geq 1$, $\beta\geq 0$ and $i=1$ or $2$,
\[b_{25}\left(2^{4\beta}3^{4\alpha}\cdot n+(i+3)\cdot 2^{4\beta}3^{4\alpha-1}-1\right)\equiv 0\pmod 5\]
and
\[b_{25}\left(2^{4\alpha}3^{4\beta}\cdot n+2^{4\alpha-1}3^{4\beta+1}-1\right)\equiv 0\pmod 5.\]
We also obtain some infinite families of congruences for the overpartition function and $t$-core partition functions by applying Theorem \ref{f1_8k_congruence} and Theorem \ref{3-congruence}.
Recall that an \emph{overpartition} of $n$ is a partition of $n$ where the first occurrence of each distinct part may be overlined \cite{Corteel-Lovejoy-2004}.
Denote by $\overline{p}(n)$ the number of overpartitions of $n$.
The congruences for $\overline{p}(n)$ have been extensively studied,
see, for example,
\cite{Mahlburg-2004, Fortin-2005, Treneer-2006, Hirschhorn-Sellers-2005, Chen-2014}.
We obtain some new infinite families of congruences for the overpartition function
modulo the powers of $2$,
 \[\overline{p}(4\cdot 3^{2\alpha}n+(4i+3)\cdot 3^{2\alpha-1})\equiv 0 \pmod{2^2},\]
\[\overline{p}(4\cdot 3^{2\alpha}n+(4i+6)\cdot 3^{2\alpha-1})\equiv 0 \pmod{2^3},\]
and
\[\overline{p}(4\cdot 3^{4\alpha}n+(4i+9)\cdot 3^{4\alpha-1})\equiv 0 \pmod{2^4},\]
where $\alpha\ge 1$, $n\ge 0$ and $i=1$ or $2$.

For $t\geq1$,
a partition is called a \emph{$t$-core partition} if none of its hook lengths is divisible by $t$.
Let $a_t(n)$ denote the number of $t$-core partitions of $n$.
For congruences for $t$-core partition functions, see, for example,
\cite{Hirschhorn-Sellers-1996-1,Hirschhorn-Sellers-1996-2,Ono-Sze-1997,Garvan-1990, Garvan-1993}.
We obtain some infinite families of congruences for $2$-core and $4$-core partition functions,
for example, for any $\alpha\ge 1, \beta\ge 0, n\ge 0$ and $i=1$ or $2$,
we have
\begin{align*}
a_2\left(3^{2\alpha} n + 3^{2\alpha-1}i+ \frac{3^{2\alpha}-1}{8}\right) \equiv 0 \pmod 2,
\end{align*}
\[a_2\left(\frac{3^{2\beta}-1}{8}\right) \equiv 1\pmod 2,\]
and
\[a_4\left(3^{2\alpha} n +  \frac{(8i+15)\cdot 3^{2\alpha-1}-5}{8}\right) \equiv 0 \pmod 2.\]
Note that this yields a special case of Hirschhorn and Seller's conjecture \cite{Hirschhorn-Sellers-1999}  when $\alpha=1$.
This conjecture has been proved by Chen \cite{Chen2013}.

The rest of this paper is organized as follows.
In Section \ref{secofmodular},
we give the  proofs of
Theorem \ref{a8k_GF} and Theorem  \ref{a3k_GF} by investigating the properties of the modular equations of the second and the third order under the Atikin $U$-operator.
In Section \ref{secofgenerating},
we first prove the existence of $\mu_m(k)$ and $\nu_m(k)$ via
linear recurrence relations \eqref{2-recurring-relations} and \eqref{3-recurring-relations}.
Combining this with the generating functions in Theorem \ref{a8k_GF} and Theorem \ref{a3k_GF},
we give the proofs of
Theorem \ref{f1_8k_congruence} and Theorem \ref{3-congruence}.
Section \ref{secofcongruences} is devoted to
infinite families of congruences for the overpartition function, $t$-core partition functions and $\ell$-regular partition functions with the aid of Theorem \ref{f1_8k_congruence} and Theorem \ref{3-congruence}.


\section{Generating functions and modular equations}\label{secofmodular}

To derive the generating functions of $p_{8k}(2^{2\alpha}n+\frac{k(2^{2\alpha}-1)}{3})$ $(1\leq k \leq 3)$
and
$p_{3k}(3^{2\beta}n+\frac{k(3^{2\beta}-1)}{8})$ $(1\leq k \leq 8)$,
we first introduce the following three operators acting on $\sum\limits_{n=-\infty}^\infty a(n)q^n$.
Let $p=2$ or $3$ and define
\begin{align*}
U_p\left\{\sum\limits_{n=-\infty}^\infty a(n)q^n\right\} & = \sum\limits_{n=-\infty}^\infty a(pn)q^{n},\\[5pt]
g_p(k)\left\{\sum\limits_{n=-\infty}^\infty a(n)q^n\right\} & = \sum\limits_{n=-\infty}^\infty a(pn+k)q^n,\\[5pt]
G_p(k)\left\{\sum\limits_{n=-\infty}^\infty a(n)q^n\right\} & = \sum\limits_{n=-\infty}^\infty a(p^2n+k)q^n.
\end{align*}
Clearly, we have $U_pg_p(k)= G_p(k)$.
Note that $U_p$ is the Atkin $U$-operator \cite{Atkin-1970},
and Cui, Gu and Huang \cite{Cui-Gu-Huang-2016} introduced the notation of $g_p$ and $G_p$ for $p=5$ or $7$.

In the notation of these three operators,
the left hand sides of \eqref{a8k_GF_1-eq}, \eqref{a8k_GF_2}, \eqref{a3k_gf} and \eqref{a3k_gf1} can be written as
\begin{align*}
     G_2^\alpha(k)\{f_1^{8k}\}&=\sum_{n=0}^\infty p_{8k}\left(2^{2\alpha} n +\frac{k(2^{2\alpha}-1)}{3}\right)q^n,\\[5pt]
     g_2G_2^\alpha(k)\{f_1^{8k}\}&=\sum_{n=-\left\lfloor k/2\right\rfloor}^\infty p_{8k}\left(2^{2\alpha+1} n +\frac{k(2^{2\alpha+2}-1)}{3}\right)q^n,\\[5pt]
     G_3^\beta(k)\{f_1^{3k}\}&=\sum_{n=0}^\infty p_{3k}\left(
    3^{2\beta}n+\frac{k(3^{2\beta}-1)}{8}\right)q^n,
\end{align*}
and
\begin{align*}
g_3G_3^\beta(k)\{f_1^{3k}\}&=\sum_{n=-\lfloor k/3 \rfloor}^\infty p_{3k}\left(3^{2\beta+1}n+\frac{k(3^{2\beta+2}-1)}{8}\right)q^n.
\end{align*}

Thus, the proof of Theorem \ref{a8k_GF} (resp. Theorem \ref{a3k_GF})
is equivalent to compute $G_2^\alpha(k)\{f_1^{8k}\}$ and
$g_2G_2^\alpha(k)\{f_1^{8k}\}$
(resp. $G_3^\beta(k)\{f_1^{3k}\}$ and $g_3G_3^\beta(k)\{f_1^{3k}\}$).

In order to obtain $G_2^\alpha(k)\{f_1^{8k}\}$ and
$g_2G_2^\alpha(k)\{f_1^{8k}\}$,
we need the following modular equation of the second order \eqref{modular_equation_2nd_order}.

Set
\[S = S(q) = \frac{f_1^{24}}{q f_2^{24}}\quad \text{and} \quad T=T(q) = \frac{f_1^8}{qf_4^8}.\]
Then the modular equation of the second order is
\begin{align}\label{modular_equation_2nd_order}
S(q^2) = T^2 + 16 T.
\end{align}
This celebrated identity was discovered by Jacobi \cite[p.~470]{Whittaker-Watson-1965},
which also appears in \cite[p.~833]{Peter-2012} and \cite[p.~394]{Hirschhorn-book}.

The next lemma shows that $U_2\{T^n\}$ is a polynomial of $S$ or $S^{-1}$.

\begin{lem}\label{h2_yield_polynomial}
For any $n\ge 0$, $U_2\{T^n\}$ is an integer polynomial in $S$ with degree at most $\lfloor \frac{n}{2}\rfloor.$
Similarly, for $n<0$, $U_2\{T^n\}$ is an integer polynomial of $S^{-1}$ with degree at most $-n$.
\end{lem}

\begin{proof}
First, we prove the case of $n\ge 0$ by induction on $n$.
For $n = 0$,
we have done since $U_2\{T^0\} = 1$.
By the following $2$-dissection of $f_1^4$ \cite[p.~40,~Entry~25]{Berndt-1991}:
\begin{align*}
f_1^4 & = \frac{f_4^{10}}{f_2^2f_8^4} - 4q\frac{f_2^2f_8^4}{f_4^2},
\end{align*}
we obtain
\[T  = \frac{f_1^8}{q f_4^8} = \frac{1}{qf_4^8} \left(\frac{f_4^{20}}{f_2^4f_8^8} - 8q f_4^8 + 16 q^2 \frac{f_2^4f_8^8}{f_4^4}\right)
= -8 + \frac{f_4^{12}}{qf_2^4f_8^8} + 16q\frac{f_2^4f_8^8}{f_4^{12}}.
\]
It implies that $U_2\{T\} = -8$. Therefore, the lemma holds for $n=0$ and $1$.
Suppose that the lemma holds for the case of $0\le n \le \lambda$ $(\lambda\ge 1).$
By means of the modular equation of the second order \eqref{modular_equation_2nd_order},
we get that
\[U_2\{T^{\lambda+1}\} = U_2\{S(q^2) T^{\lambda-1}\} - 16 U_2\{T^\lambda\} = S\cdot U_2\{T^{\lambda-1}\} - 16 U_2\{T^\lambda\},\]
which means that $U_2\{T^{\lambda+1}\}$ is also an integer polynomial of $S$.
Because $\lfloor \frac{\lambda-1}{2}\rfloor + 1 = \lfloor \frac{\lambda+1}{2}\rfloor \ge \lfloor \frac{\lambda}{2}\rfloor$,
the degree of $U_2\{T^{\lambda+1}\}$ is at most $\lfloor \frac{\lambda+1}{2}\rfloor.$
This proves the case of $n \ge 0$.

From \eqref{modular_equation_2nd_order},
we see that
\begin{align}\label{negative-rec-T}
U_2\{T^{-\lambda-1}\} = S^{-1}U_2\{T^{-\lambda+1}\} + 16S^{-1}U_2\{T^{-\lambda}\}.
\end{align}
For $n<0$, it can be shown that
\[U_2\left\{T^{-1}\right\}=8S^{-1}\quad \text{and}\quad U_2\{T^{-2}\} =128 S^{-2}+ S^{-1},\]
which imply that the lemma holds for $n=-1$ and $-2$.
Suppose that the statement holds for $-\lambda \le n \le -1$ $(\lambda \ge 2)$.
It follows from \eqref{negative-rec-T} that
$U_2\{T^{-\lambda-1}\}$ is an integer polynomial of $S^{-1}$, and its degree is at most $\lambda+1$.
\end{proof}
For ease of calculation,
we list $U_2\{T^n\}$ for $-2\leq n\leq 3$ in Table \ref{U_2_value}.

\begin{table}[htbp]
\footnotesize{
\centering
  \begin{tabular}{|p{1.6cm}<{\centering}|p{1.8cm}<{\centering}|
  p{1.8cm}<{\centering}|p{1.8cm}<{\centering}|p{1.8cm}<{\centering}|
  p{1.8cm}<{\centering}|p{1.8cm}<{\centering}|}
  \hline
  $n$ & $-2$ &$-1$& $0$& $1$ &$2$& $3$\\
  \hline
  $U_2\{T^{n}\}$ &{\scriptsize $128S^{-2}+S^{-1}$} &${8S^{-1}}$ &$1$ &$-8$ &$S+128$ &{\scriptsize $-24S-2048$}\\
  \hline
\end{tabular}
\caption{{$U_2\{T^{n}\}$ for $-2\leq n\leq 3.$}}
\label{U_2_value}}
\end{table}

For the sake of brevity,
we abbreviate $G_2(k)\{f_1^{8k}P(S)\}$ (resp. $g_2(k)\{f_1^{8k}P(S)\}$)
as $G_2\{f_1^{8k}P(S)\}$ (resp. $g_2\{f_1^{8k}P(S)\}$) if no confusions arise,
where $P(S)$ is a power series of $S$.

\begin{lem}\label{positivity-2}
For any $\alpha\ge 0$ and $1\le k\le 3$,
$f_2^{-8k}g_2G_2^{\alpha}\{f_1^{8k}\}$ is an integer polynomial of $S$
with degree $\le \lfloor\frac{k}{2}\rfloor$,
and
$f_1^{-8k}G_2^{\alpha}\{f_1^{8k}\}$ is an
integer polynomial of $S^{-1}$ with degree $\le \lfloor\frac{k}{2}\rfloor$.
\end{lem}
\begin{proof}
We prove this lemma by induction on $\alpha$.
For $\alpha=0$, we see that
\[g_2\{f_1^{8k}\} = g_2\{q^kf_4^{8k}T^k\}= f_2^{8k}U_2\{T^k\}.\]
This, together with Lemma \ref{h2_yield_polynomial},
yields that $f_2^{-8k}g_2\{f_1^{8k}\}$ is a polynomial of $S$ with integral coefficients,
and its degree is at most $\lfloor\frac{k}{2}\rfloor.$
Suppose that this lemma holds for $\alpha$,
then we set
\begin{align}\label{IU_U_2G-2}
g_2G_2^{\alpha}\{f_1^{8k}\} = f_2^{8k}\sum\limits_{i=0}^{\lfloor {k}/{2}\rfloor} a_iS^i,
\end{align}
where $a_i$ $(0\le i \le \lfloor\frac{k}{2}\rfloor)$ are integers.
Applying $U_2$ to both sides of \eqref{IU_U_2G-2},
we obtain
\begin{align}\label{G2_action_ps-1}
G_2^{\alpha+1}\{f_1^{8k}\}
= U_2g_2 G_2^{\alpha}\{f_1^{8k}\}
= U_2\left\{f_2^{8k}\sum_{i=0}^{\lfloor {k}/{2}\rfloor} a_iS^i\right\}.
\end{align}
By the definition of $U_2$,
we have for any integer $i$,
\begin{align}\label{U_2_action}
U_2\{f_2^{8k}S^i\}
 = U_2\{q^{2i}f_2^{8k-24i}f_4^{24i}T^{3i}\}
 = q^if_1^{8k-24i}f_2^{24i}U_2\{T^{3i}\}
 = f_1^{8k} S^{-i}U_2\{T^{3i}\}.
\end{align}
Combining \eqref{U_2_action} with \eqref{G2_action_ps-1},
we arrive at
\begin{align*}
  G_2^{\alpha+1}\{f_1^{8k}\}
  = f_1^{8k}\sum\limits_{i=0}^{\lfloor {k}/{2}\rfloor} a_iS^{-i}U_2\{T^{3i}\}.
\end{align*}
For $1\leq k\leq 3$,
it is immediate from Lemma \ref{h2_yield_polynomial}
that $S^{-i}U_2\{T^{3i}\}$ is an integer polynomial of $S^{-1}$ with  degree  at most $\lfloor \frac{k}{2} \rfloor$ if $0\le i \le \lfloor\frac{k}{2}\rfloor$.
So we see that
$f_1^{-8k}G_2^{\alpha+1}\{f_1^{8k}\}$ is an integer polynomial of $S^{-1}$ with degree at most
$\lfloor\frac{k}{2}\rfloor$. Since $f_1^{-8k}G_2^{0}\{f_1^{8k}\}=1$, the second statement of this lemma holds.
Thus, we assume that
\begin{equation}\label{assume_1}
 G_2^{\alpha+1}\{f_1^{8k}\} = f_1^{8k}\sum\limits_{i=0}^{\lfloor{k}/{2}\rfloor}b_iS^{-i},
\end{equation}
where
$b_i$ $(0\le i \le \lfloor\frac{k}{2}\rfloor)$
are integers.
Since
\begin{align}\label{g_2_action}
g_2\{f_1^{8k}S^{i}\} = g_2\{q^{-i}f_1^{8k+24 i}f_2^{-24i}\}
 = q^{i}f_1^{-24i}f_2^{8k+24i}U_2\{T^{k+3i}\}
 = f_2^{8k}S^{-i}U_2\{T^{k+3i}\}
\end{align}
holds for any integer $i$,
it follows from \eqref{assume_1} that
\begin{align*}
g_2G_2^{\alpha+1}\{f_1^{8k}\} = f_2^{8k}\sum\limits_{i=0}^{\lfloor{k}/{2}\rfloor}b_iS^iU_2\{T^{k-3i}\}.
\end{align*}
Note that $k-3i<0$ may hold only if $i=\lfloor\frac{k}{2}\rfloor$.
In this case, from Lemma \ref{h2_yield_polynomial},
it can be seen that the degree of $U_2\{T^{k-3i}\}$,
as an integer polynomial of $S^{-1}$, is at most $3i-k = 3\lfloor\frac{k}{2}\rfloor-k \le \lfloor \frac{k}{2}\rfloor = i.$
So $S^iU_2\{T^{k-3i}\}$ is an integer polynomial of $S$ with degree $\le \lfloor\frac{k}{2}\rfloor$.
If $k-3i \ge 0$, then the power of $S$ in $S^iU_2\{T^{k-3i}\}$ is at most $i+\lfloor\frac{k-3i}{2}\rfloor = \lfloor\frac{k-i}{2}\rfloor \le \frac{k}{2}.$
Therefore, $f_2^{-8k}g_2G_2^{\alpha+1}\{f_1^{8k}\}=\sum\limits_{i=0}^{\lfloor{k}/{2}\rfloor}b_iS^iU_2\{T^{k-3i}\}$ is an integer polynomial of $S$,
and its degree is at most $\lfloor\frac{k}{2}\rfloor$.
So the first statement is true. This completes the proof.
\end{proof}

We are ready to prove Theorem \ref{a8k_GF} based on Lemma \ref{positivity-2}.

\begin{proof}[Proof of Theorem \ref{a8k_GF}]
  We illustrate our method of proving Theorem \ref{a8k_GF} with an example of the case $k=2$.
  By Lemma \ref{positivity-2},
  we assume that for $\alpha\geq 0$
  \begin{align}\label{G_2_assumption}
    f_1^{-16} G_2^\alpha\left\{f_1^{16}\right\}
    =A_2'(2\alpha)+B_2'(2\alpha)S^{-1},
  \end{align}
  and
    \begin{align}\label{g_2G_2_assumption}
    f_2^{-16} g_2G_2^\alpha\left\{f_1^{16}\right\}
    =A_2'(2\alpha+1)+B_2'(2\alpha+1)S.
  \end{align}

  It follows from \eqref{G_2_assumption} that
  \begin{align*}
    g_2 G_2^\alpha\left\{f_1^{16}\right\}
    =A_2'(2\alpha) g_2\{f_1^{16}\}+B_2'(2\alpha) g_2\{f_1^{16}S^{-1}\}.
  \end{align*}
  By \eqref{g_2_action}, we obtain
  \begin{align*}
    g_2 G_2^\alpha\left\{f_1^{16}\right\}
    =A_2'(2\alpha)f_2^{16} U_2\{T^2\} + B_2'(2\alpha) f_2^{16} S\cdot U_2\{T^{-1}\}.
  \end{align*}
  Combining this with the expressions of $U_2\{T^2\}$ and $U_2\{T^{-1}\}$ in Table \ref{U_2_value},
  we have
  \begin{align}\label{Thm-1.1-pf-eq1}
    f_2^{-16}g_2 G_2^\alpha\left\{f_1^{16}\right\}
    =(128 A_2'(2\alpha)+8B_2'(2\alpha)) + A_2'(2\alpha)S.
  \end{align}
  Comparing \eqref{Thm-1.1-pf-eq1} with \eqref{g_2G_2_assumption}
  implies the recurrence relations
  \begin{equation}
    \begin{cases}\label{2-even-odd}
    A_2'(2\alpha+1)&=128A_2'(2\alpha)+8B_2'(2\alpha),\\[5pt]
    B_2'(2\alpha+1)&=A_2'(2\alpha).
    \end{cases}
\end{equation}
  On the other hand, multiplying $f_2^{16}$ on both sides
  of \eqref{g_2G_2_assumption} and then applying $U_2$, we obtain
     \begin{align}\label{Thm-1.1-pf-eq2}
     G_2^{\alpha+1}\left\{f_1^{16}\right\} & = A_2'(2\alpha+1) U_2\{f_2^{16}\} + B_2'(2\alpha+1) U_2\{f_2^{16} S\}.
     \end{align}
  In light of \eqref{U_2_action},
  the equality \eqref{Thm-1.1-pf-eq2} yields
  \begin{align}\label{Thm-1.1-pf-eq3}
  G_2^{\alpha+1}\left\{f_1^{16}\right\}
   = A_2'(2\alpha+1) f_1^{16} + B_2'(2\alpha+1) f_1^{16} S^{-1} U_2\{T^3\}.
  \end{align}
  Substituting the expression of $U_2\{T^3\}$ (see Table \ref{U_2_value})
  into \eqref{Thm-1.1-pf-eq3},
  we get
  \begin{align}\label{Thm-1.1-pf-eq4}
    f_1^{-16}G_2^{\alpha+1}\left\{f_1^{16}\right\}
    =\left(A_2'(2\alpha+1) -24 B_2'(2\alpha+1)\right)-2048 B_2'(2\alpha+1)S^{-1}.
  \end{align}
  Comparing \eqref{Thm-1.1-pf-eq4} with \eqref{G_2_assumption},
  we see that
  \begin{equation}
    \begin{cases}\label{2-odd-even}
    A_2'(2\alpha+2)&=A_2'(2\alpha+1)-24B_2'(2\alpha+1),\\[5pt]
    B_2'(2\alpha+2)&=-2048B_2'(2\alpha+1).
    \end{cases}
  \end{equation}
  It is obvious that $A_2'(0)=1$ and $B_2'(0)=0$.
  Hence,  $A_2'(\alpha)$ and $B_2'(\alpha)$ are determined by the recurrence relations \eqref{2-even-odd} and \eqref{2-odd-even}
  and the initial values $A_2'(0)=1$ and $B_2'(0)=0$.
  It is easy to check that $A_2'(\alpha)$ (resp. $B_2'(\alpha)$)
   and $A_2(\alpha)$ (resp. $B_2(\alpha)$) satisfy the same recurrence relation
  and share the same initial values. Thus, $A_2(\alpha)=A_2'(\alpha)$ and $B_2(\alpha)=B_2'(\alpha)$.
  This completes the proof.
\end{proof}

To study $G_3^\beta(k)\{f_1^{3k}\}$ and
$g_3G_3^\beta(k)\{f_1^{3k}\}$ $(1\le k \le 8)$,
we require the following modular equation of the third order \eqref{modular_equation_3rd_order}.

Define
\begin{align*}
  X = X(q) = \frac{f_1^{12}}{qf_3^{12}}\quad \text{and}\quad Y = Y(q) =\frac{f_1^3}{qf_9^3}.
\end{align*}
Then the modular equation of the third order is of the form
\begin{align}\label{modular_equation_3rd_order}
X(q^3) = Y^3 + 9 Y^2 + 27 Y.
\end{align}
For more details, please refer to
\cite[Theorem 9.1]{Cooper-2017} and \cite[p.~397]{Hirschhorn-book}.

Now we proceed to derive some properties related to the modular equation of the third order.

\begin{lem}\label{h3_yield_polynomial}
For any $n\ge 0$,
$U_3\{Y^n\}$ is an integer polynomial in $X$ with degree at most $\lfloor \frac{n}{3}\rfloor.$
Similarly, for $n<0$,
$U_3\{Y^n\}$ is an integer polynomial of $X^{-1}$ with degree at most $-n$.
\end{lem}
\begin{proof}
We first show the statements of the case of $n\geq0$ by induction on $n$.
Due to Hirschhorn \cite[p.~397,~ (43.3.14)]{Hirschhorn-book},
it is easy to deduce that
\begin{align}\label{U_3-first_values}
  U_3\{Y^0\} = 1,\qquad U_3\{Y\}=-3\qquad\text{ and }\qquad U_3\{Y^2\}=9.
\end{align}
Therefore, the lemma holds for $n=0$, $1$ and $2$.
Suppose that the lemma holds for $0\le n \le \lambda$ $(\lambda\ge 2)$.
By means of the modular equation of the third order \eqref{modular_equation_3rd_order},
we get that
\[U_3\{Y^{\lambda+1}\} = X\cdot U_3\{Y^{\lambda-2}\} - 27 U_3\{Y^{\lambda-1}\} - 9U_3\{Y^\lambda\},\]
which means that $U_3\{Y^{\lambda+1}\}$ is also an integer polynomial of $X$.
Because $\lfloor \frac{\lambda-2}{3}\rfloor + 1 = \lfloor \frac{\lambda+1}{3}\rfloor \ge \lfloor \frac{\lambda}{3}\rfloor \ge \lfloor \frac{\lambda-1}{3}\rfloor$,
the degree of $U_3\{Y^{\lambda+1}\}$ is at most $\lfloor \frac{\lambda+1}{3}\rfloor.$
This proves the case of $n \ge 0$.

From \eqref{modular_equation_3rd_order},
we derive that
\begin{align}\label{negative-rec-U}
U_3\{Y^{-\lambda-1}\} =X^{-1}U_3\{Y^{-\lambda+2}\} +9 X^{-1}U_3\{Y^{-\lambda+1}\}+ 27X^{-1}U_3\{Y^{-\lambda}\}.
\end{align}
For $n < 0$, combining \eqref{negative-rec-U} with \eqref{U_3-first_values},
we have
\begin{align*}
  U_3\{Y^{-1}\}&=9X^{-1},\\[5pt]
  U_3\{Y^{-2}\} &= 6X^{-1}+243X^{-2},\\[5pt]
  U_3\{Y^{-3}\}&=X^{-1}+234X^{-2}+6561X^{-3},
\end{align*}
which show that the lemma holds for $n=-1$, $-2$ and $-3$.
Suppose that the statement holds for $-\lambda \le n \le -1$ $(\lambda \ge 3)$.
From \eqref{negative-rec-U},
we derive that
$U_3\{Y^{-\lambda-1}\}$ is also an integer polynomial of $X^{-1}$,
and its degree is at most $\lambda+1$.
This completes the proof for $n<0$.
\end{proof}
For ease of calculation,
we list $U_3\{Y^n\}$ for $-3\leq n\leq 8$ in Table \ref{U_3_value}.

\begin{table}[hbtp]
\scriptsize\begin{tabular}{|p{1cm}<{\centering}|c|
  p{1cm}<{\centering}|c|}
  \hline
   $n$& $U_3\{Y^{n}\}$ & $n$& $U_3\{Y^{n}\}$\\
  \hline
  $-3$&$6561X^{-3}+243X^{-2}+X^{-1}$&
  $3$&$X$\\
  \hline
  $-2$&${243}{X^{-2}}+6{X^{-1}}$ & $4$&$-12 X-243$    \\
  \hline
  $-1$&$9{X^{-1}}$ & $5$&$90 X+2187$    \\
  \hline
  $0$&$1$ & $6$&$X^2-486 X-13122$ \\
  \hline
  $1$&$-3$ & $7$&$-21 X^2+1701 X+59049$ \\
  \hline
  $2$&$9$&$8$&$252 X^2-177147$\\
  \hline
\end{tabular}
\caption{$U_3\{Y^{n}\}$ for $-3\leq n\leq 8.$}
\label{U_3_value}
\end{table}
For the sake of convenience, we abbreviate $G_3(k)\{f_1^{3k}P(X)\}$ (resp. $g_3(k)\{f_1^{3k}P(X)\}$)
as $G_3\{f_1^{3k}P(X)\}$ (resp. $g_3\{f_1^{3k}P(X)\}$) if no confusions arise,
where $P(X)$ is a power series of $X$.

We also obtain the following equalities
similar to  \eqref{U_2_action} and \eqref{g_2_action},
for any integer $i$,
\begin{align}\label{U_3_action}
  U_3\{f_3^{3k}X^i\}
= U_3\{q^{-i}f_1^{12i}f_3^{3k-12i}\}
 = q^if_1^{3k-12i}f_3^{12i}U_3\{Y^{4i}\}
 = f_1^{3k} X^{-i}U_3\{Y^{4i}\},
\end{align}
and
\begin{align}\label{g_3_action}
  g_3\{f_1^{3k}X^{i}\}
= g_3\{q^{-i}f_1^{3k+12 i}f_3^{-12i}\}
= q^{i}f_1^{-12i}f_3^{3k+12i}U_3\{Y^{k+4i}\}
= f_3^{3k}X^{-i}U_3\{Y^{k+4i}\}.
\end{align}
Based on the above two equalities,
we can derive the following statements parallel to Lemma \ref{positivity-2}.

\begin{lem}\label{positivity-3}
For any $\beta\ge 0$ and $1\le k\le 8$,
$f_3^{-3k}g_3G_3^{\beta}\{f_1^{3k}\}$ is an integer polynomial of $X$ with degree $\le \lfloor\frac{k}{3}\rfloor$, and
 $f_1^{-3k}G_3^{\beta}\{f_1^{3k}\}$ is an integer polynomial of $X^{-1}$ with degree $\le \lfloor\frac{k}{3}\rfloor$.
\end{lem}
\begin{proof}
  The proof is analogous to the proof of Lemma \ref{positivity-2},
  and hence is omitted.
\end{proof}

We prove Theorem \ref{a3k_GF} with the aid of Lemma \ref{positivity-3}.

\begin{proof}[Proof of Theorem \ref{a3k_GF}]
Since there are eight cases should be considered for $1\leq k \leq 8$,
we only prove this theorem for $k=7$ without loss of generality.
Other cases can be derived similarly.
By Lemma \ref{positivity-3}, we assume that for any $\beta\ge 0$,
\begin{align}\label{G3_assumption}
f_1^{-21} G_3^{\beta}\{f_1^{21}\} = C_7'(2\beta) + D'_{7,1}(2\beta)X^{-1} + D'_{7,2}(2\beta)X^{-2},
\end{align}
and
\begin{align}\label{g3G3_assumption}
f_3^{-21} g_3G_3^{\beta}\{f_1^{21}\} = C_7'(2\beta+1) + D'_{7,1}(2\beta+1)X + D'_{7,2}(2\beta+1)X^{2}.
\end{align}
Multiplying $f_1^{21}$ on both sides of   \eqref{G3_assumption} and applying $g_3$, we have
\begin{align*}
g_3G_3^{\beta}\{f_1^{21}\} =C_7'(2\beta) g_3\{f_1^{21}\} + D'_{7,1}(2\beta)g_3\{f_1^{21}X^{-1}\}+D'_{7,2}(2\beta)g_3\{f_1^{21}X^{-2}\}.
\end{align*}
Combining this with \eqref{g_3_action}, we obtain that
\begin{align}\label{Thm-1.2-pf-eq1}
g_3G_3^{\beta}\{f_1^{21}\}=
C'_7(2\beta)f_3^{21}U_3\{Y^7\} + D'_{7,1}(2\beta)f_3^{21}XU_3\{Y^{3}\}+D'_{7,2}(2\beta)f_3^{21}X^2U_3\{Y^{-1}\}.
\end{align}
Substituting the expressions of $U_3\{Y^{-1}\}$, $U_3\{Y^3\}$ and $U_3\{Y^7\}$ (see Table \ref{U_3_value}) into \eqref{Thm-1.2-pf-eq1},
we get
\begin{align}
f_3^{-21}g_3G_3^{\beta}\{f_1^{21}\}
&=59049 C'_7(2\beta) + (1701C'_7(2\beta)+9D'_{7,2}(2\beta))X\nonumber \\[5pt] &\qquad+(-21C'_7(2\beta)+D'_{7,1}(2\beta))X^2.\label{Thm-1.2-pf-eq2}
\end{align}
Matching the coefficients of like powers of $X$ in \eqref{Thm-1.2-pf-eq2} and  \eqref{g3G3_assumption},
we obtain
\begin{equation}
\begin{cases}
\label{3-even-odd}
  C'_7(2\beta+1) &=59049C'_7(2\beta),\\[5pt]
  D'_{7,1}(2\beta+1)&=1701C'_7(2\beta)+9D'_{7,2}(2\beta),\\[5pt]
  D'_{7,2}(2\beta+1)&=-21 C'_7(2\beta)+D'_{7,1}(2\beta).
\end{cases}
\end{equation}
Meanwhile, the equality  \eqref{g3G3_assumption} implies that
\begin{align*}
G_3^{\beta+1}\left\{f_1^{21}\right\} = C'_7(2\beta+1)U_3\left\{f_3^{21}\right\}
+ D'_{7,1}(2\beta+1)U_3\{f_3^{21}X\}
+ D'_{7,1}(2\beta+1)U_3\{f_3^{21}X^2\}.
\end{align*}
By \eqref{U_3_action}, we have
\begin{align*}
G_3^{\beta+1}\left\{f_1^{21}\right\} = C'_7(2\beta+1)f_1^{21}
+ D'_{7,1}(2\beta+1)f_1^{21}X^{-1}U_3\{Y^4\}
+ D'_{7,2}(2\beta+1)f_1^{21}X^{-2}U_3\{Y^8\}.
\end{align*}
Due to the expressions of $U_3\{Y^4\}$ and $U_3\{Y^8\}$
(see Table \ref{U_3_value}), we arrive at
\begin{align*}
G_3^{\beta+1}\left\{f_1^{21}\right\} = &
\left(C'_7(2\beta+1)-12D'_{7,1}(2\beta+1)+252D'_{7,2}(2\beta+1)\right)f_1^{21}
\\[5pt]
&\qquad-243D'_{7,1}(2\beta+1)f_1^{21}X^{-1}-177147D'_{7,2}(2\beta+1)f_1^{21} X^{-2}.
\end{align*}
Combining this with \eqref{G3_assumption}, we see that
\begin{equation}
\begin{cases}
\label{3-odd-even}
  C'_7(2\beta+2)&=C'_7(2\beta+1)-12D'_{7,1}(2\beta+1)+252D'_{7,2}(2\beta+1),\\[5pt]
  D'_{7,1}(2\beta+2)&=-243D'_{7,1}(2\beta+1),\\[5pt]
  D'_{7,2}(2\beta+2)&=-177147D'_{7,2}(2\beta+1).
\end{cases}
\end{equation}
Clearly, $C'_7(0)=1$, $D'_{7,1}(0)=D'_{7,2}(0)=0$. Hence, $C'_7(\beta)$, $D'_{7,1}(\beta)$ and $D'_{7,2}(\beta)$ are uniquely determined by \eqref{3-even-odd}, \eqref{3-odd-even} and the initial values $C'_7(0)=1$, $D'_{7,1}(0)=D'_{7,2}(0)=0$.
By direct calculations,
we know that $C'_7(\beta)$, $D'_{7,1}(\beta)$
and $D'_{7,2}(\beta)$ enjoy the same recurrence relation
and share the same initial values as $C_7(\beta)$, $D_{7,1}(\beta)$
and $D_{7,2}(\beta)$, respectively. Thus,
for any $\beta\ge 0$, we have $C'_7(\beta)=C_7(\beta)$, $D'_{7,1}(\beta)=D_{7,1}(\beta)$ and $D'_{7,2}(\beta)=D_{7,2}(\beta)$.
This completes the proof.
\end{proof}

\section{Congruences for $p_{8k}(n)$ and $p_{3k}(n)$}\label{secofgenerating}

In this section, we first prove that for any $m\geq 2$,
$\mu_m(k)$ $(1\leq k \leq 3)$ and $\nu_m(k)$ $(1\leq k \leq 8)$ always exist.
Then we establish infinite families of congruences modulo $m$ for
$p_{8k}(n)$ $(1\le k \le 3)$ and $p_{3k}(n)$ $(1\le k \le 8)$ based on the
generating functions in Theorem \ref{a8k_GF} and Theorem \ref{a3k_GF}.

Recall that
\begin{align*}
  \mu_m(k)=\min\{\alpha\geq 1\mid B_k(2\alpha-1)\equiv 0\pmod m \},
\end{align*}
and
\begin{align*}
   \nu_m(k)=\min\left\{\beta\geq 1\mid D_{k,i}(2\beta-1)\equiv 0\pmod m\ \text{ for }\ 1 \leq i \leq \left\lfloor\frac{k}{3}\right\rfloor\right\},
\end{align*}
where
$B_k(\alpha)$ and $
D_{k,i}(\beta)$ are obtained in
Theorem \ref{a8k_GF} and Theorem \ref{a3k_GF}, respectively.

Now we point out that for any $m\geq 2$,
$\mu_m(k)$ $(1\leq k \leq 3)$ and $\nu_m(k)$ $(1\leq k \leq 8)$
must exist. In order to prove the existence, we need the definition of period of
 recurring sequences. For details, refer to \cite[Chapter 6]{Lidl-1986}.

Let $\{s_n\}_{n=0}^{\infty}$ be an integer sequence.
If there exist positive integers
$r $ and $n_0 $ such that
$s_{n+r} \equiv s_n \pmod m $ for all $n \geq n_0$,
then the sequence is called \emph{ultimately periodic for the modulus $m$} and $r$
is called a \emph{period} of the sequence.
The smallest number among all the
possible periods of an ultimately periodic sequence for the modulus $m$
is called the \emph{least period of the sequence}.
It is easy to prove that every period of an ultimately periodic sequence for the modulus $m$ is divisible by the least period.
An ultimately periodic sequence $\{s_n\}_{n=0}^{\infty}$ for the modulus $m$
with least period $r$ is called \emph{periodic} if $s_{n+r} \equiv s_n \pmod{m}$ holds for all  $n\geq 0$.

Let $d$ be a positive integer and $a_0,\ldots,a_{d-1} $ be given
integers with $a_0\neq 0$.
An integer sequence $s_0, s_1,\ldots$ is called a  $d$th-order \emph{linear recurring sequence} if it satisfies the relation
\begin{align*}
  s_{n+d}=a_{d-1}s_{n+d-1}+a_{d-2}s_{n+d-2}
  +\cdots+a_0s_n,\quad \text{for }n \geq 0
\end{align*}
with the initial values $s_0,s_1,\cdots, s_{d-1}$.

Engstrom \cite{Engstrom1931} studied the periodic property of linear recurring sequences for modulus $m$.
\begin{lem}{\rm\cite[p.~210]{Engstrom1931}}\label{Eng-lem}
  Let $s_0, s_1,\ldots$ be a $d$th-order linear recurring sequence satisfying
  \begin{align*}
  s_{n+d}=a_{d-1}s_{n+d-1}+a_{d-2}s_{n+d-2}
  +\cdots+a_0s_n,\quad \text{for }n \geq 0
  \end{align*}
  with $a_0\neq0$ and the initial values $s_0,s_1,\ldots, s_{d-1}$.
  If $m$ is prime to $a_0$, then the
  $d$th-order linear recurring sequence $s_0,s_1,\ldots$ is periodic;
  otherwise, it is ultimately periodic for the modulus $m$.
\end{lem}

We have proved that $\left\{B_k(2\alpha)\right\}_{\alpha=0}^\infty$
(resp. $\{D_{k,i}(2\beta)\}_{\beta=0}^\infty$) is a linear recurring sequence in Theorem \ref{a8k_GF} (resp. Theorem \ref{a3k_GF}).
Using the above result of Engstrom,
we prove that for arbitrary $m\geq2$,
$\mu_m(k)$ (resp. $\nu_m(k)$) defined by \eqref{def_mu} (resp. \eqref{def_nu}) always exists.

\begin{thm}\label{existence_of_mu_nu}
For the sequence $\{B_k(2\alpha-1)\}_{\alpha=1}^\infty$
(resp. $\{D_{k,i}(2\beta-1)\}_{\beta=1}^\infty$) obtained in Theorem \ref{a8k_GF} (resp. \ref{a3k_GF}),
the corresponding $\mu_m(k)$ (resp. $\nu_m(k)$)  always exists for any $m\geq 2$.
\end{thm}

\begin{proof}
We only prove the case of $\nu_m(k)$,
and the case of $\mu_m(k)$ is completely similar.
It is sufficient to show there exists a positive integer $\beta_0$ such that $D_{k,i}(2\beta_0-1)\equiv0\pmod{m}$ for all $1\leq i\leq \lfloor\frac{k}{3}\rfloor$.
According to Theorem \ref{a3k_GF} and Table \ref{coeff_initial-3},
we observe that the corresponding $a_0$ of these linear recurring sequences $\{D_{k,i}(2\beta)\}_{\beta=0}^\infty$ are of the form $\pm3^t$ for some integer $t$.
Therefore, by Lemma \ref{Eng-lem},
if $m$ is coprime to $3$,
then the integer sequences $\{D_{k,i}(2\beta)\}_{\beta=0}^\infty$ are periodic for the modulus $m$.
This implies that there exist positive integers $\beta_0$
such that $D_{k,i}(2\beta_0)\equiv 0\pmod m$ since $D_{k,i}(0)=0$.
Observe that if $1\leq i\leq \lfloor\frac{k}{3}\rfloor$ then $-3^wD_{k,i}(2\beta-1)=D_{k,i}(2\beta)$ for some positive integer $w$.
So we deduce that $D_{k,i}(2\beta_0-1)\equiv0\pmod{m}$ for all $1\leq i\leq \lfloor\frac{k}{3}\rfloor$.

It remains to show that $\nu_m(k)$ also exists when $m$ is not prime to $3$.
 Assume that $m=3^j\cdot \ell$ with  $3\nmid \ell$ and $j\ge 1$. We discuss the property of $D_{k,i}(2\beta-1)$ modulo $3^j$ and $\ell$, respectively.
 First, it should be noticed that for any $1\leq i\leq \lfloor\frac{k}{3}\rfloor$,
 there always exists a positive integer $N_i$ such that for any $\beta\geq N_i$,
 \begin{align}\label{re-cong-1}
  D_{k,i}(2\beta-1)\equiv0 \pmod{3^j}.
 \end{align}
This is because all the coefficients $h(k)$ and $r(k)$ in the recurrence relation \eqref{3-recurring-relations} are divisible by $3$.
Second, by Lemma \ref{Eng-lem} again, we see that the integer sequences $\{D_{k,i}(2\beta)\}_{\beta=0}^\infty$ are periodic for the modulus $\ell$.
Let $r_i$ be the least period of the $\{D_{k,i}(2\beta)\}_{\beta=0}^\infty$  modulo $\ell$.
Then the least common multiple $r$ of $r_i$ $(1\leq i\leq \lfloor\frac{k}{3}\rfloor)$
 is a period of all sequences $\{D_{k,i}(2\beta)\}_{\beta=0}^\infty$, which implies
\begin{align*}
  0\equiv D_{k,i}(0)\equiv D_{k,i}(2r)\equiv D_{k,i}(4r)\equiv D_{k,i}(6r)\equiv\cdots\pmod{\ell}, \quad \text{for}\ 1\leq i \leq \left\lfloor\frac{k}{3}\right\rfloor.
\end{align*}
Combining this with the fact that if $1\leq i\leq \lfloor\frac{k}{3}\rfloor$ then $-3^wD_{k,i}(2\beta-1)=D_{k,i}(2\beta)$ for some positive integer $w$, we have
\begin{align}\label{re-cong-2}
  0\equiv D_{k,i}(2r-1)\equiv D_{k,i}(4r-1)\equiv D_{k,i}(6r-1)\equiv\cdots\pmod{\ell}, \quad \text{for}\ 1\leq i \leq \left\lfloor\frac{k}{3}\right\rfloor.
\end{align}
Let $n_0$ be an integer such that  $n_0 r\geq N=\max_{1\leq i \leq \lfloor\frac{k}{3}\rfloor}{N_i}$ and set $\beta_0=n_0 r$.
Combining \eqref{re-cong-1} with \eqref{re-cong-2}, we have
$D_{k,i}(2\beta_0-1)\equiv0 \pmod{m}.$
Therefore, we conclude that $\nu_m(k)$ must exist.
Similarly, we can show that $\mu_m(k)$ always exists for any $m\geq 2$.
This completes the proof.
\end{proof}

Recall that $c_1$  is  the integer such that $0\leq c_1\leq m-1$ and $c_1\equiv A_k(2\mu_m(k)-1)\pmod m$ and  $c_2$  is  the integer such that $0\leq c_2\leq m-1$ and $c_2\equiv C_k(2\nu_m(k)-1)\pmod m$.
According to Theorem \ref{existence_of_mu_nu}, it follows from \eqref{a8k_GF_1-eq} and \eqref{a3k_gf} that for any $m\geq 2$,
\[g_2G_2^{\mu_m(k)-1}\{f_1^{8k}\} \equiv c_1\cdot f_2^{8k} \pmod m,\]
and
\[g_3G_3^{\nu_m(k)-1}\{f_1^{3k}\} \equiv c_2\cdot f_3^{3k} \pmod m.\]
Thus we can establish infinite families of congruences for
$p_{8k}(n)$ $(1\le k \le 3)$ and $p_{3k}(n)$ $(1\le k \le 8)$.

\begin{proof}[Proof of Theorem \ref{f1_8k_congruence}]
For any $m\geq2$,
based on Theorem \ref{existence_of_mu_nu},
we have
\[g_2G_2^{\mu_m(k)-1}\{f_1^{8k}\} \equiv c_1\cdot f_2^{8k} \pmod m,\]
where $0\leq c_1\leq m-1$ and $c_1 \equiv A_k(2\mu_m(k)-1) \pmod m$.
Acting  the operator $U_2$, we obtain that
\[G_2^{\mu_m(k)}\{f_1^{8k}\} \equiv c_1\cdot f_1^{8k} \pmod m.\]
Then for $\alpha \ge 1,$ the following congruence holds
\[G_2^{\mu_m(k)(\alpha-1)}\{f_1^{8k}\} \equiv c_1^{\alpha-1} \cdot f_1^{8k} \pmod m.\]
Acting the operator $g_2G_2^{\mu_m(k)-1}$, we have
\[g_2G_2^{\mu_m(k)\alpha-1}\{f_1^{8k}\} \equiv c_1^{\alpha} \cdot f_2^{8k} \pmod m.\]
Equivalently,
\begin{align}
\sum\limits_{n=0}^\infty p_{8k}\left(2^{2\mu_m(k)\alpha-1}n + \frac{k(2^{2\mu_m(k)\alpha}-1)}{3}\right)q^n \equiv c_1^\alpha f_2^{8k} \pmod m. \label{f1_8k_cong_reason}
\end{align}
As a consequence, if $c_1^\alpha \not\equiv 0 \pmod m$ for any $\alpha\ge 1$, then we have
\begin{align*}
& p_{8k}\left(2^{2u_m(k)\alpha-1}(2n+1) + \frac{k(2^{2u_m(k)\alpha}-1)}{3}\right)\\[5pt]
= ~ & p_{8k}\left(2^{2u_m(k)\alpha}n + \frac{(2k+3)\cdot2^{2u_m(k)\alpha-1}-k}{3}\right)\\[5pt]
\equiv ~ & 0 \pmod{m}.
\end{align*}
Otherwise, there exists a positive integer $\alpha$ such that $c_1^{\alpha} \equiv 0 \pmod m$.
Then for such positive integer $\alpha$,
it follows from \eqref{f1_8k_cong_reason} that
\[p_{8k}\left(2^{2\mu_m(k)\alpha-1}n + \frac{k(2^{2\mu_m(k)\alpha}-1)}{3}\right) \equiv 0 \pmod m.\]
This completes the proof.
\end{proof}
Comparing the constant terms on both sides of \eqref{f1_8k_cong_reason} yields Corollary \ref{f1-8k-cor}.

\begin{cor}\label{f1-8k-cor}
Let $m$ and $k$ be integers with $m\geq 2$ and $1\leq k\leq3$.
Then for any $\alpha\geq 1$,
\begin{align*}
p_{8k}\left(\frac{k\cdot 2^{2\mu_m(k)\alpha}-k}{3}\right) \equiv c_1^\alpha \pmod m.
\end{align*}
\end{cor}

We are now ready to complete the proof of Theorem \ref{3-congruence}.

\begin{proof}[Proof of Theorem \ref{3-congruence}]
For any $m\geq2$,
based on Theorem \ref{existence_of_mu_nu},
\[g_3G_3^{\nu_m(k)-1}\{f_1^{3k}\} \equiv c_2\cdot f_3^{3k} \pmod m,\]
where $0\leq c_2\leq m-1$ and $c_2 \equiv C_k(2\nu_m(k)-1) \pmod m$.
Acting  the operator $U_3$, we obtain that
\[G_3^{\nu_m(k)}\{f_1^{3k}\} \equiv c_2\cdot f_1^{3k} \pmod m.\]
Then for $\beta \ge 1,$ we get
\[G_3^{\nu_m(k)(\beta-1)}\{f_1^{3k}\} \equiv c_2^{\beta-1} \cdot f_1^{3k} \pmod m.\]
Acting  the operator $g_3G_3^{\nu_m(k)-1}$, we have
\[g_3G_3^{\nu_m(k)\beta-1}\{f_1^{3k}\} \equiv c_2^{\beta} \cdot f_3^{3k} \pmod m,\]
that is,
\begin{align}
\sum\limits_{n=0}^\infty p_{3k}\left(3^{2\nu_m(k)\beta-1}n + \frac{k(3^{2\nu_m(k)\beta}-1)}{8}\right)q^n \equiv c_2^\beta\cdot f_3^{3k}\pmod m. \label{3-cong-reason}
\end{align}
Consequently, if $c_2^\beta \not\equiv 0 \pmod m$ for any $\beta\ge 1$, then we have
\begin{align*}
& p_{3k}\left(3^{2\nu_m(k)\beta-1}(3n+i) + \frac{k(3^{2\nu_m(k)\beta}-1)}{8}\right)\\[5pt]
= ~ & p_{3k}\left(3^{2\nu_m(k)\beta}n
+ \frac{(3k+8i)3^{2\nu_m(k)\beta-1}-k}{8}\right)\\[5pt]
\equiv ~ & 0 \pmod{m},
\end{align*}
where  $i=1$ or $2$.
Otherwise,
let $\beta$ be an integer satisfying $c_2^{\beta} \equiv 0 \pmod m$,
then we have
\[p_{3k}\left(3^{2\nu_m(k)\beta-1}n + \frac{k(3^{2\nu_m(k)\beta}-1)}{8}\right) \equiv 0 \pmod m\]
and hence the proof is complete.
\end{proof}

Equating the constants on both sides of \eqref{3-cong-reason} directly implies Corollary \ref{f1-3k-cor}.

\begin{cor}\label{f1-3k-cor}
Let $m$ and $k$ be integers with $m\geq2$ and $1\leq k\leq 8$.
Then for any $\beta\geq1$,
\begin{align*}
p_{3k}\left(\frac{k\cdot 3^{2\nu_m(k)\beta}-k}{8}\right) \equiv c_2^\beta \pmod m.
\end{align*}
\end{cor}


\section{Congruences for certain partition functions}\label{secofcongruences}

In this section,
we derive infinite families of congruences for certain partition functions,
such as the overpartition function,
$t$-core partition functions,
$\ell$-regular partition functions
in light of  Theorems \ref{a8k_GF},  \ref{a3k_GF}, \ref{f1_8k_congruence} and \ref{3-congruence}.
The values of $\mu_m(k)$ and $\nu_m(k)$ play an important role in deducing congruences,
and one can refer to Table \ref{values_u_m(k)} for the values about some small primes $m$.

\subsection{Congruences for the overpartition function}

Recall that an overpartition of $n$ is a partition of $n$,
where the first occurrence of each distinct part may be overlined. We denote the number of overpartitions of $n$ by $\overline{p}(n)$.
The generating function of $\overline{p}(n)$ is
\begin{equation*}
  \sum_{n=0}^{\infty}\overline{p}(n)q^n=\frac{(-q;q)_\infty}{(q;q)_\infty}=\frac{f_2}{f^2_1}.
\end{equation*}
In order to obtain the congruences of $\overline{p}(n)$,
we employ a 4-dissection of $f_2/f^2_1$.
This identity was proved by Andrews, Passary, Sellers and Yee \cite{Andrews2016}. Here we give a new  proof.

\begin{lem}\label{f2f1-4-dissection1}
We have
\begin{equation*}
\frac{f_2}{f^2_1}=\frac{f_{8}^{19}}{f_4^{14}f_{16}^6}
+2q\frac{f_8^{13}}{f_4^{12}f_{16}^2}+4q^2\frac{f_8^7f_{16}^2}{f_4^{10}}
  +8q^3\frac{f_8f_{16}^6}{f_4^8}.
\end{equation*}
\end{lem}

\begin{proof}
Employing the 2-dissection \cite[p.~40, Entry 25]{Berndt-1991} of $1/f_1^2$
\[\frac{1}{f_1^2}=\frac{f_8^5}{f_2^5f_{16}^2}+2q\frac{f_4^2f_{16}^2}{f_2^5f_8},\]
 we deduce that
\begin{align*}
  \frac{f_2}{f_1^2}
=f_2\left(
\frac{f_8^5}{f_2^5f_{16}^2}+2q\frac{f_4^2f_{16}^2}{f_2^5f_8}\right)
=\frac{1}{f_2^4}\left(\frac{f_8^5}{f_{16}^2}+2q\frac{f_4^2f_{16}^2}{f_8}\right).
\end{align*}
Thanks to the 2-dissection \cite[p.~40, Entry 25]{Berndt-1991} of $1/f_1^4$
\[\frac{1}{f_1^4}=\frac{f_4^{14}}{f_2^{14}f_{8}^4}
+4q\frac{f_4^2f_{8}^4}{f_2^{10}},\]
we derive that
\begin{align*}
  \frac{f_2}{f_1^2}
&=\left(\frac{f_8^{14}}{f_4^{14}f_{16}^4}
+4q^2\frac{f_8^2f_{16}^4}{f_4^{10}}\right)\left(\frac{f_8^5}{f_{16}^2}
+2q\frac{f_4^2f_{16}^2}{f_8}\right)\\[5pt]
&=\frac{f_8^{19}}{f_4^{14}f_{16}^6}+2q\frac{f_8^{13}}{f_4^{12}f_{16}^2}
+4q^2\frac{f_8^7f_{16}^2}{f_4^{10}}+8q^3\frac{f_8f_{16}^6}{f_4^8}.
\end{align*}
This completes the proof.
\end{proof}

The generating functions of
$\overline{p}(4n+k)$ $(k=0,1,2\ \text{and}\ 3)$ immediately follow from Lemma~\ref{f2f1-4-dissection1}.
In fact, Fortin,
Jacob and Mathieu \cite{Fortin-2005}, and Hirschhorn and Sellers \cite{Hirschhorn-Sellers-2005}
established the
$2$-, $3$-, and $4$-dissections
of the generating function for $\overline{p}(n)$.

\begin{cor}\label{generatingoverpartitions}
  We have
  \begin{align*}
    \sum_{n=0}^\infty \overline{p}(4n)q^n   &= \frac{f_2^{19}}{f_1^{14}f_4^6},\\[5pt] \sum_{n=0}^{\infty}\overline{p}(4n+1)q^n&=2\frac{f_2^{13}}{f_1^{12}f_4^{2}},\\[5pt]
    \sum_{n=0}^{\infty}\overline{p}(4n+2)q^n&=4\frac{f_2^7f_{4}^{2}}{f_1^{10}},
  \end{align*}
  and
  \begin{align*}
  \sum_{n=0}^{\infty}\overline{p}(4n+3)q^n&=8\frac{f_2 f_{4}^{6}}{f_1^{8}}.
  \end{align*}
\end{cor}
So we deduce that
\begin{equation}\label{overpartition_4n_1}
  \frac{\sum_{n=0}^{\infty}\overline{p}(4n+1)q^n}{2}\equiv f_1^6\pmod 2,
\end{equation}
\begin{equation}\label{overpartition_4n_2}
  \frac{\sum_{n=0}^{\infty}\overline{p}(4n+2)q^n}{4}\equiv f_1^{12}\pmod 2,
\end{equation}
and
\begin{equation}\label{overpartition_4n_3}
  \frac{\sum_{n=0}^{\infty}\overline{p}(4n+3)q^n}{8}\equiv f_1^{18}\pmod 2.
\end{equation}

Applying Theorem \ref{3-congruence} to the above identities, we obtain some congruences for  $\overline{p}(n)$.

\begin{thm}\label{overpartition-cong}
For any $\alpha\ge 1$, $n\ge 0$ and $i=1$ or $2$,
we have
\begin{description}
  \item[$\mathrm{(1)}$] $\overline{p}(4\cdot 3^{2\alpha}n+(4i+3)\cdot 3^{2\alpha-1})\equiv 0 \pmod{2^2};$
  \item[$\mathrm{(2)}$] $\overline{p}(4\cdot 3^{2\alpha}n+(4i+6)\cdot 3^{2\alpha-1})\equiv 0 \pmod {2^3};$
  \item[$\mathrm{(3)}$] $\overline{p}(4\cdot 3^{4\alpha}n+(4i+9)\cdot 3^{4\alpha-1})\equiv 0 \pmod {2^4}$.
\end{description}
\end{thm}
\begin{proof}
\begin{description}
    \item[$\mathrm{(1)}$]
    For any $n\geq0$,
    it follows from \eqref{overpartition_4n_1} that
    \begin{align}\label{overpatition_a6}
    \frac{\overline{p}(4n+1)}{2} \equiv p_6(n) \pmod 2.
    \end{align}
    Due to  Theorem \ref{a3k_GF},
    we deduce that $\nu_2(2)=1$ and
    $g_3\{f_1^{6}\} \equiv f_3^{6} \pmod 2.$
    Combining this with Theorem \ref{3-congruence},
    we see that for any  $\alpha \ge 1$ and $ n\ge 0$,
    the congruences
    \[p_6\left(3^{2\alpha}n+3^{2\alpha-1} i + \frac{3^{2\alpha}-1}{4}\right) \equiv 0 \pmod 2\]
    hold for  $i=1$ and $2$.
    Then by \eqref{overpatition_a6} we arrive at
 \[\overline{p}\left(4\cdot 3^{2\alpha}n+(4i+3)\cdot 3^{2\alpha-1}\right)\equiv 0 \pmod{2^2}.\]
    \item[$\mathrm{(2)}$]
    By \eqref{overpartition_4n_2},
    we have for any $n\ge 0$,
\begin{align}\label{overpatition_a12}
\frac{\overline{p}(4n+2)}{4} \equiv p_{12}(n) \pmod 2.
\end{align}
Due to  Theorem \ref{a3k_GF},
we deduce that $\nu_2(4)=1$ and
$g_3\{f_1^{12}\} \equiv f_3^{12} \pmod 2.$
Combining this with Theorem \ref{3-congruence},
we derive that for any $\alpha \ge 1$, $n\ge 0$ and $i=1$ or $2$,
\[p_{12}\left(3^{2\alpha}n+3^{2\alpha-1} i + \frac{3^{2\alpha}-1}{2}\right) \equiv 0 \pmod 2.\]
Substituting this into \eqref{overpatition_a12} yields
\[\overline{p}\left(4\cdot 3^{2\alpha}n+(4i+6)\cdot 3^{2\alpha-1}\right)\equiv 0 \pmod {2^3}.\]
    \item[$\mathrm{(3)}$]
    In terms of \eqref{overpartition_4n_3},
    we have for any $n\ge 0$,
\begin{align}\label{overpatition_a18}
\frac{\overline{p}(4n+3)}{8} \equiv p_{18}(n) \pmod 2.
\end{align}
Due to  Theorem \ref{a3k_GF},
we deduce that $\nu_2(6)=2$ and $g_3G_3\{f_1^{18}\} \equiv f_3^{18} \pmod 2.$
In light of Theorem \ref{3-congruence},
we obtain that for any $\alpha \ge 1$, $n\ge 0$ and $i=1$ or $2$,
\[p_{18}\left(3^{4\alpha}n+3^{4\alpha-1} i +  \frac{3^{4\alpha+1}-3}{4}\right) \equiv 0 \pmod 2.\]
Substituting this into \eqref{overpatition_a18} yields
\[\overline{p}\left(4\cdot 3^{4\alpha}n+(4i+9)\cdot 3^{4\alpha-1}\right)\equiv 0 \pmod {2^4}.\]
\end{description}
This completes the proof.
\end{proof}

\begin{rem}
  Yang, Cui and Lin \cite{Yang-2017} proved some infinite families of congruences modulo powers of $2$ for $\overline{p}(n)$. For example, for $\alpha\geq 0$, $n\geq 0$ and $i=1$ or  $2$, they showed that
  \begin{align*}
    \overline{p}(8\cdot3^{4\alpha+4}n+(4i+9)\cdot 3^{4\alpha+3})\equiv 0\pmod{2^8},
  \end{align*}
  which contains the case of $\mathrm{(3)}$ in Theorem \ref{overpartition-cong}
  when $n$ is even.
\end{rem}

\subsection{Congruences for $t$-core partition functions}
Next we will use  Theorem \ref{3-congruence} to prove some congruences for $t$-core partition functions.
The following generating function of the number of $t$-core partitions of $n$ is given by Garvan, Kim
and Stanton \cite{Garvan-1990},
\[\sum_{n=0}^{\infty}a_{t}(n)q^n=\frac{f_{t}^{t}}{f_1}.\]
\begin{thm}
For any $\alpha\ge 1, \beta\ge 0, n\ge 0$ and $i=1$ or $2$,
\begin{align}\label{2-core-congruence}
a_2\left(3^{2\alpha} n + 3^{2\alpha-1}i+ \frac{3^{2\alpha}-1}{8}\right) \equiv 0 \pmod 2
\end{align}
and
\[a_2\left(\frac{3^{2\beta}-1}{8}\right) \equiv 1\pmod 2.\]
\end{thm}

\begin{proof}
Since $\sum\limits_{n=0}^\infty a_2(n)q^n = \frac{f_2^2}{f_1} \equiv f_1^3 \pmod 2$,
it follows that for any $n\ge 0$,
\begin{align*}
a_2(n)\equiv p_3(n) \pmod 2.
\end{align*}
By Theorem \ref{a3k_GF},
we have $\nu_2(1)=1$ and $g_3\{f_1^3\} \equiv f_3^3 \pmod 2$.
Then Theorem \ref{3-congruence} implies that for any $\alpha\ge 1$ and $n\ge 0$,
\[p_3\left(3^{2\alpha} n + 3^{2\alpha-1}i+ \frac{3^{2\alpha}-1}{8}\right) \equiv 0 \pmod 2,\]
where  $i=1$ or $2$. Then the congruence \eqref{2-core-congruence} holds as claimed.
By Corollary \ref{f1-3k-cor}, we have for any $\beta\ge 0$,
\[a_2\left(\frac{3^{2\beta}-1}{8}\right) \equiv p_3\left(\frac{3^{2\beta}-1}{8}\right)
 \equiv 1\pmod 2.\]
 This completes the proof.
\end{proof}

Hirschhorn and Sellers \cite{Hirschhorn-Sellers-1999} conjectured that for any $t\geq 2$ and $k=0, 2$,
\[a_{2^t}\left(\frac{3^{2^{t-1}-1}(24n+8k+7)-\frac{4^t-1}{3}}{8}\right)\equiv 0\pmod 2\]
holds for all $n\geq 0$,
which has been proved by Chen \cite{Chen2013}.
We can also deduce congruences for
the $4$-core partition function.
More specifically,
we have for any $\alpha\ge 1$, $\beta\geq 0$, $n\ge 0$ and $i=1$ or $2$
\begin{align}\label{4-core-congruence}
  a_4\left(3^{2\alpha} n +  \frac{(8i+15)\cdot 3^{2\alpha-1}-5}{8}\right) \equiv 0 \pmod 2,
\end{align}
and
\begin{align*}
a_4\left(\frac{5\cdot 3^{2\beta}-5}{8}\right) \equiv 1\pmod 2.
\end{align*}
It is obvious that
setting $\alpha=1$ in
\eqref{4-core-congruence} yields Hirschhorn and Seller's conjecture for $t=2$.
Moreover, Gao, Cui and Guo \cite[Theorem 1.2]{Gao-2017} contains the case of $i=1$.

\subsection{Congruences for $\ell$-regular partition functions}

Now we study arithmetic properties for $\ell$-regular partition functions.
Recall that the generating function of $b_\ell(n)$ is given by
\[\sum_{n=0}^{\infty}b_\ell(n)q^n=\frac{f_\ell}{f_1}.\]

\begin{thm}\label{25regular}
For any $\alpha\ge 1$, $\beta\geq0$, $n\ge 0$ and $i=1$ or $2$,
\begin{align}\label{25-regular-2-cong}
b_{25}\left(2^{4\alpha}n + 3\cdot 2^{4\alpha-1}-1\right) \equiv 0 \pmod 5,
\end{align}
\begin{align}\label{25-regular-3-cong}
b_{25}\left(3^{4\alpha}n + (i+3)\cdot 3^{4\alpha-1}-1\right) \equiv 0 \pmod 5,
\end{align}
 and
\begin{align*}
b_{25}\left(3^{4\beta}-1\right) \equiv 1\pmod 5.
\end{align*}
\end{thm}

\begin{proof}
It is easy to see that
\begin{equation}\label{a24-2-identity3}
\sum\limits_{n=0}^\infty b_{25}(n)q^n \equiv f_1^{24} \pmod 5.
\end{equation}
By Theorem \ref{a8k_GF}  we find that $\mu_5(3)=2$ and
\begin{align}\label{a24-2-identity}
g_2G_2\{f_1^{24}\} \equiv f_2^{24} \pmod 5.
\end{align}
By Theorem \ref{a3k_GF} we obtain  that $\nu_5(8)=2$ and
\begin{align}\label{a24-3-identity}
g_3G_3\{f_1^{24}\} \equiv f_3^{24} \pmod 5.
\end{align}
Then Theorem \ref{f1_8k_congruence} and Theorem \ref{3-congruence} imply that \eqref{25-regular-2-cong} and \eqref{25-regular-3-cong} hold, respectively.

Combining  \eqref{a24-3-identity} with Corollary \ref{f1-3k-cor}, we find that  for any $\beta \ge 0$,
\[b_{25}\left(3^{4\beta}-1\right)\equiv 1 \pmod 5.\]
This completes the proof.
\end{proof}

By considering Theorem \ref{f1_8k_congruence} and Theorem \ref{3-congruence} together,  we obtain a generalized form of Theorem \ref{25regular}.
\begin{cor}
  For any $\alpha\geq 1$, $\beta\geq 0$ and $i=1$ or $2$, we have
  \[b_{25}\left(2^{4\beta}3^{4\alpha}\cdot n+(i+3)\cdot 2^{4\beta}3^{4\alpha-1}-1\right)\equiv 0\pmod 5,\]
  and
  \[b_{25}\left(2^{4\alpha}3^{4\beta}\cdot n+2^{4\alpha-1}3^{4\beta+1}-1\right)\equiv 0\pmod 5.\]
\end{cor}
\begin{proof}
By \eqref{a24-2-identity}  and \eqref{a24-3-identity}, we find that for any $\beta\ge 0$,
\begin{align*}
G_2^{2\beta}\{f_1^{24}\} \equiv f_1^{24} \pmod 5\quad \text{and}\quad
G_3^{2\beta}\{f_1^{24}\} \equiv f_1^{24} \pmod 5.
\end{align*}
It follows from \eqref{a24-2-identity3} that
\begin{equation}\label{a24-2-identity4}
b_{25}\left(2^{4\beta} n+2^{4\beta}-1\right)\equiv b_{25}(n)\pmod 5
\end{equation}
and
\begin{equation}\label{a24-3-identity4}
b_{25}\left(3^{4\beta} n+3^{4\beta}-1\right)\equiv b_{25}(n)\pmod 5.
\end{equation}
Substitute \eqref{25-regular-2-cong} and \eqref{25-regular-3-cong} into \eqref{a24-3-identity4} and \eqref{a24-2-identity4}, respectively and  the proof follows.
\end{proof}

We also obtain some congruences for other  $\ell$-regular partition functions.
Andrews et al. \cite[Theorem 3.5]{Andrews2010} showed that for $\alpha\geq0$ and $n\geq0$,
\begin{align}
  b_4\left(3^{2\alpha+2}n+\frac{11\cdot3^{2\alpha+1}-1}{8}\right)&\equiv0\pmod{2},\label{b4-cong-2}\\[5pt]
  b_4\left(3^{2\alpha+2}n+\frac{19\cdot3^{2\alpha+1}-1}{8}\right)&\equiv0\pmod{2}.\label{b4-cong-3}
\end{align}
The congruences \eqref{b4-cong-2} and \eqref{b4-cong-3} can be easily derived by Theorem \ref{3-congruence}.
Since the proofs are similar to the proof of Theorem \ref{25regular},
we  list the congruences for
$9$-, $17$-, $19$-regular partition functions without proof.
For any  $\alpha\ge 1$, $\beta\ge 0$, $n\ge 0$ and $i=1$ or $2$,
we have
\begin{align}\label{9-regular}
b_9\left(2^{2\alpha}n+\frac{5\cdot 2^{2\alpha-1}-1}{3}\right) \equiv 0 \pmod 3,
\end{align}
\begin{align*}
b_9\left(\frac{2^{2\beta}-1}{3}\right)\equiv 1 \pmod 3,
\end{align*}
\begin{align*}
b_{17}\left(2^{18\alpha}n+\frac{7\cdot2^{18\alpha-1}-2}{3}\right) \equiv 0 \pmod{17},
\end{align*}
and
\begin{align*}
b_{19}\left( 3^{10\alpha}n+\frac{(4i+9)\cdot 3^{10\alpha-1}-3}{4}\right) \equiv 0 \pmod {19}.
\end{align*}
Note that the congruence \eqref{9-regular}  is coincident  with Keith \cite[Theorem 3]{Keith-2014}.

\vskip 0.2cm

\noindent{\bf Acknowledgements.}
This work was supported by the 973 Project and the National Science Foundation of China.

\appendix
\section*{Appendix}

\begin{table}[H]
\centering
\caption{The values of $f(k)$, $g(k)$ and the initial values of
$A_k(\alpha)$ and $B_k(\alpha)$.}
\begin{tabular}{|p{0.5cm}<{\centering}|c|c|c|}
  \hline
  $k$ & $f(k)$ & $g(k)$ & Initial values \\
  \hline
  1 & $-2^3$ & 0 & \makecell[cc]{$A_1(0)=1$;
  $A_1(1)=-8$.}\\
  \hline
  2 & $2^3\cdot13$ & $-2^{14}$ & \makecell[cc]{$A_2(0)=1, A_2(2)=104;$
  $A_2(1)=128, A_2(3)=-3072;$\\ $B_2(0)=0, B_2(2)=-2048;$
  $B_2(1)=1, B_2(3)=104.$}\\
  \hline
  3 & $-2^6\cdot5\cdot11$ & $-2^{22}$ & \makecell[cc]{$A_3(0)=1, A_3(2)=-1472;$
  $A_3(1)=-2048, A_3(3)=3014656;$\\ $B_3(0)=0, B_3(2)=49152;$
  $B_3(1)=-24, B_3(3)=84480.$}\\
  \hline
\end{tabular}
\label{coeff_initial}
\end{table}

\begin{table}[H]
\centering
\caption{The values of $h(k)$, $r(k)$
and the initial values of $C_k(\beta)$ and $D_{k,i}(\beta)$.}
\footnotesize\begin{tabular}{|p{0.5cm}<{\centering}|c|c|c|}
  \hline
  $k$ & $h(k)$ & $r(k)$  & Initial values\\
  \hline
  1   & -3     &   0        & \makecell[cc]{$C_1(0)=1,C_1(1)=-3$.}  \\
  \hline
  2   & $3^2$      &   0         & \makecell[cc]{$C_2(0)=1,C_2(1)=9.$}   \\
  \hline
  3   & $-2^2\cdot 3$    & $-3^7$      & \makecell[cc]{$C_3(0)=1, C_3(2)=-12$;
  $C_3(1)=0,C_3(3)=-2187$;\\
  $D_{3,1}(0)=0, D_{3,1}(2)=-243$;
  $D_{3,1}(1)=1, D_{3,1}(3)=-12$.} \\
  \hline
  4   & $-2\cdot 3^2 \cdot 19$   & $-3^{10}$      & \makecell[cc]{$C_4(0)=1, C_4(2)=-99$;
  $C_4(1)=-243, C_4(3)=24057$;\\
  $D_{4,1}(0)=0,D_{4,1}(2)=2916$;
  $D_{4,1}(1)=-12,D_{4,1}(3)=4104$.}  \\
 \hline
  5   & $2^2\cdot 3^3\cdot 17$   &$-3^{13}$      &
  \makecell[cc]{
  $C_5(0)=1,C_5(2)=1107$;
  $C_5(1)=2187,C_5(3)=2421009$;\\
  $D_{5,1}(0)=0,D_{5,1}(2)=-21870$;
  $D_{5,1}(1)=90,D_{5,1}(3)=165240$.\\
  }\\
  \hline
  6   & $-2\cdot 3^2\cdot 17\cdot23$ &$-3^{16}$ & \makecell[cc]{$C_6(0)=1,C_6(2)=-7038$;
  $C_6(1)=-13122,C_6(3)=49305915$;\\
  $D_{6,1}(0)=0,D_{6,1}(2)=118098$;
  $D_{6,1}(1)=-486,D_{6,1}(3)=3420468$;\\
  $D_{6,2}(0)=0,D_{6,2}(2)=-177147$;
  $D_{6,2}(1)=1,D_{6,2}(3)=-7038$.\\}\\
  \hline
  7   & $2^2\cdot3^3\cdot 491$  &$-3^{19}$  & \makecell[cc]{
  $C_7(0)=1,C_7(2)=33345$;
  $C_7(1)=59049,C_7(3)=1968988905$;\\
  $D_{7,1}(0)=0,D_{7,1}(2)=-413343$;
  $D_{7,1}(1)=1701,D_{7,1}(3)=90200628$;\\
  $D_{7,2}(0)=0,D_{7,2}(2)=3720087$;
  $D_{7,2}(1)=-21,D_{7,2}(3)=-1113588$.\\
}\\
  \hline
  8   & $- 2\cdot3^4\cdot5\cdot359$  &$-3^{22}$  & \makecell[cc]{
  $C_8(0)=1,C_8(2)=-113643$;
  $C_8(1)=-177147,C_8(3)=20131516521$;\\
  $D_{8,2}(0)=0,D_{8,2}(2)=-44641044$;
  $D_{8,2}(1)=252,D_{8,2}(3)=-73279080$.\\
}\\
  \hline
\end{tabular}
\label{coeff_initial-3}
\end{table}

\begin{table}[H]
  \centering
  \fontsize{9}{8}\selectfont
    \begin{tabular}{|c|c|c|c||c|c|c|c|c|c|c|c|}
    \hline
    \multirow{2}{*}{$m$}&
    \multicolumn{3}{c||}{$\mu_m(k)$}&\multicolumn{8}{c|}{$\nu_m(k)$}
    \cr\cline{2-12}
    &$k=1$&$k=2$&$k=3$&$k=1$&$k=2$&$k=3$&$k=4$&$k=5$&$k=6$&$k=7$&$k=8$\cr
    \hline
    2 & 1 & 2 & 1 &1 &  1 &  2  &  1  &  1   &   2  &  2   & 1\cr\hline
   3 & 1 & 3 & 1 & 1 & 1  &  2  &  1  &   1  &   2  &   1  & 1\cr\hline
   5 & 1 & 5 &  2& 1 & 1  &   4 &  3  &    1 &    5 &   4  & 2\cr\hline
   7 & 1 &  4&  7& 1 & 1  &   8  &  4  &    8 &   7  &   1  & 1\cr\hline
   11 & 1 &  5&  2& 1 & 1 &   5  &  3  &     5&     5&     5& 11\cr\hline
   13 & 1 &  2&  7& 1 & 1 &   7  &  3  &     6&     13&   7  &7\cr\hline
   17 & 1 &   9& 17& 1 & 1  &  16 &  9  &     2&     2&    16 &8\cr\hline
   19& 1 &  19&  10& 1 & 1 &  20  &  2  &     20&     5&  20   & 9\cr
    \hline
    \end{tabular}
  \caption{Values for $\mu_m(k)$ and $\nu_m(k)$.}
  \label{values_u_m(k)}
\end{table}

\end{document}